\documentclass[12pt, reqno]{amsart}
\usepackage{latexsym}
\usepackage{amsfonts}
\usepackage{amsmath}
\usepackage{fourier}

\usepackage{color}              
\usepackage{hyperref}
\usepackage{enumerate, amsfonts, amsthm}

\newcommand{\ch}[1]{#1}

\usepackage{amssymb}
\usepackage[textsize=small]{todonotes}       
\usepackage{nicefrac}

\def\bl{\begin{lemma}}
\def\el{\end{lemma}}
\def\bth{\begin{theorem}}
\def\eth{\end{theorem}}
\def\bc{\begin{corollary}}
\def\ec{\end{corollary}}
\def\bcj{\begin{conjecture}}
\def\ecj{\end{conjecture}}
\def\bpr{\begin{proposition}}
\def\epr{\end{proposition}}
\def\bde{\begin{definition}}
\def\ede{\end{definition}}
\def\E{\mathbb{E}}

\newcommand{\be}{\begin{eqnarray}}
\newcommand{\ee}{\end{eqnarray}}
\newcommand{\eps}{{\mbox{$\epsilon$}}}

\newcommand{\Z}{{\mathbb Z}}

\newcommand{\M}{{\mathcal M}}

\newcommand{\Vertices}{V(G)}
\newcommand{\Edges}{E(G)}
\newcommand{\ct}[1]{ \stackrel{#1}{\conn} }

\newcommand{\Tm}{T_{{\rm mix}}}
\newcommand{\C}{{\mathcal{C}}}

\newcommand{\prob}{\mbox{\bf P}}

\newcommand{\p}{\mbox{\bf p}}

\newcommand{\lr}{\leftrightarrow}

\newtheorem{theorem}{Theorem}[section]
\newtheorem{definition}{Definition}[section]
\newtheorem{lemma}[theorem]{Lemma}

\newtheorem{corollary}[theorem]{Corollary}
\newtheorem{proposition}[theorem]{Proposition}

\newtheorem{conjecture}[theorem]{Conjecture}

\theoremstyle{definition}
\numberwithin{equation}{section}

\def\arrowfillCS#1#2#3#4{%
   \thickmuskip0mu\medmuskip\thickmuskip\thinmuskip\thickmuskip
   \relax#4#1\mkern-7mu%
   \cleaders\hbox{$#4\mkern-2mu#2\mkern-2mu$}\hfill
   \mkern-7mu#3
}
\def\lrfill{\arrowfillCS\leftarrow\relbar\rightarrow\relax}

\newcommand{\mnot}{\Tm}

\newcommand{\lrr}{\stackrel{r}{\lr}}

\renewcommand{\and}{\hbox{ {\rm and} }}

\newcommand{\off}{\hbox{ {\rm off} }}

\newcommand{\vep}{\varepsilon}
\renewcommand{\eps}{\varepsilon}
\newcommand{\epsm}{\eps_{\m}}
\newcommand{\epsn}{\eps_{n}}

\newcommand{\whp}{whp}

\newcommand{\sss}{\scriptscriptstyle}
\newcommand{\e}{{\mathrm e}}
\newcommand {\convd}{\stackrel{d}{\longrightarrow}}
\newcommand {\convp}{\stackrel{\sss {\mathbb P}}{\longrightarrow}}

\newcommand{\nn}{\nonumber}

\newcommand{\Park}{P_{r,r_0}}
\newcommand{\conn}{\longleftrightarrow}
\newcommand{\lbeq}[1]  {\label{e:#1}}
\newcommand{\refeq}[1] {\eqref{e:#1}}
\newcommand{\eqn}[1]{\begin{equation} #1 \end{equation}}
\newcommand{\eqan}[1]{\begin{align} #1 \end{align}}

\newcommand{\cluster}{{\mathcal C}}

\newcommand{\expec}{{\mathbb E}}

\def\1{{\mathchoice {1\mskip-4mu\mathrm l}      
{1\mskip-4mu\mathrm l}
{1\mskip-4.5mu\mathrm l} {1\mskip-5mu\mathrm l}}}
\newcommand{\indic}[1]{\1_{\{#1\}}}

\newcommand{\m}{{m}}

\input epsf.sty

\begin{document}

\title{Unlacing the lace expansion:
a survey to hypercube percolation}
\author{Remco van der Hofstad}
\address{Department of Mathematics and
	Computer Science, Eindhoven University of Technology, P.O.\ Box 513,
	5600 MB Eindhoven, The Netherlands.}
\email{rhofstad@win.tue.nl}
\author{Asaf Nachmias}
\address{Department of Mathematics, University of British Columbia, 121-1984 Mathematics Rd, Vancouver, BC, Canada V6T1Z2.}
\email{asafnach@math.ubc.ca}


\begin{abstract} The purpose of this note is twofold. First, we survey the study of the percolation phase transition on the Hamming hypercube $\{0,1\}^\m$ obtained in the series of papers \cite{BCHSS1, BCHSS2, BCHSS3, HN12}. Secondly, we explain how this study can be performed without the use of the so-called ``lace-expansion'' technique. To that aim, we provide a novel simple proof that the triangle condition holds at the critical probability. We hope that some of these techniques will be useful to obtain non-perturbative proofs in the analogous, yet much more difficult study on high-dimensional tori.



%
\end{abstract}

\maketitle

\section{Introduction}
\label{sec-intro}

This paper is intended to be a companion to the papers \cite{BCHSS1,BCHSS2,BCHSS3,HN12} in the setting of percolation on the Hamming hypercube $\{0,1\}^m$. Our goal here is to present the most recent state of affairs of this topic, emphasizing ideas, techniques and the gaps in our understanding. We will present one novel proof of a result obtained in \cite{BCHSS2}, namely, that the triangle condition on the hypercube holds at the critical probability. This proof is simpler than the one presented in \cite{BCHSS2} as it does not use the {\em lace-expansion} technique. Given this estimate, it will be indicated in this note how the study of the qualitative properties of the phase transition on the hypercube can be obtained {\em without} the use of the lace expansion. While this has yielded a non-perturbative proof in the hypercube setting, ``unlacing'' the proofs in the setting of high-dimensional tori seems much more difficult and requires new ideas. We hope that the hypercube study will get us closer to this goal. \\


\ch{The paper is organized as follows. In Section \ref{sec-gnp}, we describe the phase transition in
the Erd\H{o}s and R\'{e}nyi random graph, our main source of inspiration and in Section \ref{sec-phase-trans-ncube} we describe the analogous results, obtained in \cite{BCHSS1, BCHSS2, HN12}, in the setting of the hypercube. We proceed in Section \ref{sec-bchss} to introduce and discuss the role of the so-called {\em triangle condition} for percolation. This condition arises fairly naturally in the study of percolation and to exemplify this, we present a classical argument of Barsky and Aizenman \cite{BA91} showing how to control the expected cluster size using the triangle condition. Next, in Section \ref{sec-rwconditions} we present some conditions about the behavior of the random walk on the underlying graph and state a general theorem that allows us to analyze the phase transition in percolation on any graph satisfying these conditions. In Section \ref{sec-backtohypercube} we restrict our attention back to the hypercube setting and verify that the random walk conditions of the previous section holds. In particular, we state there the required estimate on $p_c$ (Theorem \ref{thm-unlacing}) which we prove in this paper. We conclude this chapter in Section \ref{sec-open} by discussing open problems.

Section \ref{sec-overview} contains an overview of the proof of Theorem \ref{thm-super} and Section \ref{sec-unlacing} provides a proof of Theorem \ref{thm-unlacing} which together with the argument in Section \ref{sec-unif-conn-NBW} yields a simple proof that the triangle condition holds on the hypercube.}

\subsection{The Erd\H{o}s and R\'{e}nyi random graph}
\label{sec-gnp}
Recall that $G(n,p)$ is obtained from the complete graph by retaining
each edge of the complete graph on $n$ vertices with probability $p$
and erasing it otherwise, independently for all edges. Write $\C_j$ for
the $j$-th largest component obtained this way. An inspiring discovery of
Erd\H{o}s and R\'{e}nyi \cite{ER} is that this model exhibits a phase transition
when $p$ is scaled like $p=c/n$.
When $c<1$ we have $|\C_1|=\Theta(\log n)$ \whp\ and
$|\C_1|=\Theta(n)$ \whp{} when $c>1$. Here, we say that a sequence of events $E_n$
occurs \emph{with high probability} (\whp) when $\lim_{n\rightarrow \infty}\prob(E_n)=1$.
We further write $f(\m)=O(g(\m))$ if $|f(\m)|/|g(\m)|$ is uniformly bounded
from above by a positive constant, $f(\m)=\Theta(g(\m))$ if $f(\m)=O(g(\m))$ and
$g(\m)=O(f(\m))$, $f(\m)=\Omega(g(\m))$ if $1/f(\m)=O(1/g(\m))$ and $f(\m)=o(g(\m))$ if $f(\m)/g(\m)$ tends to $0$ with $\m$.

The investigation of the case $c$ close to $1$, initiated by Bollob\' as \cite{B1}
and further studied by {\L}uczak \cite{Lu},
revealed an intricate picture of the phase transition's nature. See
\cite{Bol85} for results up to 1984, and \cite{Aldo97,JKL93,JLR00,LPW94} for
references to subsequent work. We briefly describe these now. \\

\noindent {\bf \em The subcritical phase.} Let $\epsn=o(1)$ be a non-negative
sequence with $\epsn \gg n^{-1/3}$ and put $p=(1-\epsn)/n$, then, for any
fixed integer $j\geq 1$,
    $$
    \frac{|\C_j|}{2\epsn^{-2} \log(\epsn^3 n)}\convp 1\, ,
    $$
where $\convp$ denotes convergence in probability.
\medskip

\noindent {\bf \em The critical window.} When $p=(1 + O(n^{-1/3}))/n$,
for any fixed integer $j\geq 1$,
    $$
    \Big( {|\C_1| \over n^{2/3}}, \ldots, {|\C_j| \over n^{2/3}} \Big )
    \stackrel{d}{\longrightarrow} (\chi_1, \ldots, \chi_j) \, ,
    $$
where $(\chi_i)_{i=1}^j$ are random variables supported on $(0,\infty)$,
and $\convd$ denotes convergence in distribution.\newline

\noindent {\bf \em The supercritical phase.} Let $\epsn=o(1)$ be a non-negative sequence with
$\epsn \gg n^{-1/3}$ and put $p=(1+\epsn)/n$, then
    $$
    \frac{|\C_1|}{2\epsn n}\convp 1 \, ,
    $$
while, for any fixed integer $j \geq 2$,
    $$
    \frac{|\C_j|}{2\epsn^{-2} \log(\epsn^3 n)}\convp 1\, .
    $$
\medskip
%

\subsection{The phase transition on the hypercube}
\label{sec-phase-trans-ncube}


One inherent difficulty for percolation on the hypercube, \ch{or on any finite graph,} is that it is not obvious how to \emph{define} its critical value. \ch{In the Erd\H{o}s and R\'{e}nyi random graph this critical probability should turn out to be $(1+O(n^{-1/3}))/n$ in whichever definition we use!} When $p=(1+O(n^{-1/3}))/n$ in $G(n,p)$ we have that the mean cluster size is of order $n^{1/3}$. This inspired Borgs, Chayes, the first author, Slade and Spencer \cite{BCHSS1, BCHSS2, BCHSS3} to suggest that the precise location $p_c=p_c(\lambda)$ of the phase transition is the unique solution to the equation
    \be
    \label{pcdef}
    \E_{p_c} |\C(0)| = \lambda 2^{\m/3} \, .
    \ee
where $\C(0)$ is the connected component containing the origin, $|\C(0)|$ denotes its size, and $\lambda\in(0,1)$ denotes an arbitrary constant that is typically taken to be small. Here $2^{\m/3}$ can be viewed as the cube root of the volume of the graph, i.e., its number of vertices. \ch{There are several other more intuitive definitions (see the discussion in Section 7 of \cite{NP3}), however, in order to justify these definitions one needs to show that analogous results to the ones described in Section \ref{sec-gnp} holds with this definition. To the best of our knowledge this was only done with (\ref{pcdef}). Let us now describe the phase transition of percolation on the hypercube around this $p_c$.} \\
%
%
%

From here on, we take $p_c=p_c(\lambda)$ with $\lambda\in(0,1)$
a fixed constant. The phase transition on the hypercube is described in the
following three theorems, in all of which we consider bond percolation on the hypercube $\{0,1\}^\m$ with varying $p$.

\begin{theorem}[The subcritical phase \cite{BCHSS1, BCHSS2}]
\label{thm-sub}
Put $p=p_c(1-\epsm)$ where $\epsm=o(1)$ is a \ch{positive} sequence satisfying $\epsm \gg 2^{-\m/3}$. Then,
for all fixed $\delta>0$,
    	\eqn{
    	\prob_p\big(\epsm^{-2}/3600 \leq |\C_1|
	\leq (2+\delta)\epsm^{-2} \log(\epsm^3 2^{\m})\big) = 1-o(1) \, ,
    	}
and
	\eqn{
	\label{chi-asy-sub}
	\E_p|\C(0)|=\frac{1+o(1)}{|\epsm|}.
	}
\end{theorem}

\begin{theorem}[The critical window \cite{BCHSS1, BCHSS2}]
\label{thm-crit}
Put $p=p_c(1+\epsm)$ with $|\epsm|= O(2^{-\m/3})$. Then,
    	\eqn{
	\prob_p (\omega^{-1} 2^{2\m/3} \leq |\C_1| \leq \omega 2^{2\m/3})\geq  1 - O(\omega^{-1}) \, .
    	}
and
	\eqn{
	\label{chi-asy-pc}
	 \E_p |\C(0)| = \Theta(2^{\m/3}) \, .
	}
\end{theorem}

\begin{theorem}[The supercritical phase \cite{HN12}]
\label{thm-super}
Put $p=p_c(1+\epsm)$ where $\epsm=o(1)$ is a positive
sequence with $\epsm \gg 2^{-\m/3}$. Then
    \eqn{
    \frac{|\C_1|}{2\epsm 2^\m}\convp 1\, ,
    }
where $\convp$ denotes convergence in probability, and
    \eqn{
    \E_p |\C(0)| = (4+o(1))\epsm^2 2^\m \, .
    }
Furthermore, the second largest component $\C_2$ satisfies
    \eqn{
    \frac{|\C_2|}{\epsm 2^m}\convp 0 \, .
    }
\end{theorem}

\ch{Theorems \ref{thm-sub} and \ref{thm-crit} are proved in \cite{BCHSS1, BCHSS2}. The work in \cite{BCHSS1, BCHSS2} did not provide sharp estimates for the supercritical phase and the authors conjectured (see Conjecture 3.2 in \cite{BCHSS3}) the statement of Theorem \ref{thm-super}, proved in \cite{HN12}. Thus, Theorems \ref{thm-sub}--\ref{thm-super} fully identify the phase transition and the critical window in the hypercube.}

%
%


\subsection{The role of the percolation triangle condition}
\label{sec-bchss}
Let us briefly review the study of random subgraphs of general finite transitive graphs
initiated in \cite{BCHSS1, BCHSS2}. Let $G$ be a finite transitive graph
and write $V$ for the number of vertices of $G$ and $\m$ for its degree.
Let $p\in [0,1]$ and write $G_p$ for the random graph obtained from $G$
by retaining each edge with probability $p$ and erasing it with probability $1-p$,
independently for all edges. We also write $\prob_p$ for this probability measure.
We say an edge is $p$-open ($p$-closed) if was retained
(erased). We say that a path in the graph is
$p$-open if all of its edges are $p$-open. For two
vertices $x,y$ we write $x \lr y$ for the event that there exists a $p$-open path
connecting $x$ and $y$. For an integer $j \geq 1$ we write $\C_j$ for the $j$-th largest
component of $G_p$ (breaking ties arbitrarily) and for a vertex $v$ we write $\C(v)$
for the component in $G_p$ containing $v$.

For two vertices $x,y$ we denote
    \be \label{nabledef}
    \nabla_p(x,y) = \sum_{u,v} \prob_p(x \lr u)\prob_p(u \lr v)\prob_p(v \lr y) \, .
    \ee
The quantity $\nabla_p(x,y)$, known as the {\em triangle diagram}, was introduced by
Aizenman and Newman \cite{AN} to study critical percolation on high-dimensional infinite
lattices. In that setting, the important feature of an infinite graph $G$ is
whether $\nabla_{p_c}(0,0)<\infty$.
This condition is often referred to as the {\em triangle condition}.
In high-dimensions, Hara and Slade \cite{HS90}
proved \ch{that} the triangle condition \ch{holds}. It allows to deduce
that numerous critical exponents attain the same values as they do on an infinite regular
tree, see e.g.\ \cite{ AN, BA91, KN08, KN11}.

When $G$ is a finite graph, $\nabla_p(0,0)$ is obviously finite, however,
there is still a {\em finite} triangle condition which in turn guarantees that
random critical subgraphs of $G$ have the same geometry as random
subgraphs of the complete graph on $V$ vertices, where $V$ denotes
the number of vertices in $G$. That is, in the finite setting the role of
the infinite regular tree is played by the complete graph. Let us make this
heuristic formal.

We always have that $V\to \infty$ and that $\lambda\in(0,1)$ is a
fixed and small constant. Let $p_c=p_c(\lambda)$ be defined by
    \be
    \label{pcdefgeneral}
    \E _{p_c(\lambda)} |\C(0)| = \lambda V^{1/3} \, .
    \ee
The finite triangle condition is the assumption that
$\nabla_{p_c(\lambda)}(x,y) \leq {\bf 1}_{\{x=y\}} + a_0$,
for some $a_0=a_0(\lambda)$ sufficiently small.
The strong triangle condition, defined in  \cite[(1.26)]{BCHSS2},
is the statement that there exists a constant $C$ such that
for all $p\leq p_c$,
    \be
    \label{finitetriangle}
    \nabla_{p}(x,y) \leq {\bf 1}_{\{x=y\}} + {C\chi(p)^3 \over V} + o(1),
    \ee
where $o(1)$ tends to $0$ as $\m\to \infty$.
In \cite{BCHSS2}, \eqref{finitetriangle} is shown to hold for
various graphs: the complete graph,
the hypercube and high-dimensional tori $\Z_n^d$. Its proof relies on the \emph{lace expansion},
a perturbative technique to investigate the two-point function $\prob_{p_c}(0\conn x)$ that was first used for percolation on the high-dimensional infinite lattice by Hara and Slade \cite{HS90}.
The lace expansion is an extremely powerful technique, but is also quite involved.
We feel that it should only be used when more elementary techniques fail.
Apart from surveying the literature on hypercube percolation, our aim in this paper
is to show that  Theorems \ref{thm-sub}-\ref{thm-super} can be proved \emph{without}
relying on the lace expansion.

\ch{The main result of \cite{BCHSS1} is that the triangle condition implies strong estimates on $|\C_1|$ in the critical and subcritical case:}

\begin{theorem}[\cite{BCHSS1}]
\label{thm-BCHSS1}
Consider bond percolation on a finite transitive graph $G$ having $V$ vertices and
degree $\m$ satisfying the strong triangle condition \eqref{finitetriangle} where
$\m\rightarrow \infty$ as $V\rightarrow \infty$.
Then the assertions of Theorems \ref{thm-sub} and \ref{thm-crit} hold when each
occurrence of $2^{\m}$ is replaced with $V$.
\end{theorem}

The triangle condition is significant since it arises naturally in various calculation one performs. To indicate how this occurs, we provide here a classical argument of Barsky and Aizenman \cite{BA91} controlling the size of $\E_p|\C(0)|$ as $p$ varies. We put $\chi(p)=\expec_p[|\cluster(0)|]$ and will show that the strong triangle condition \eqref{finitetriangle} implies that for all $p<p_c$,
	\eqn{
	\lbeq{chi(p)-asy}
	\chi(p)=\frac{1+O(\max_{v\neq 0}\nabla_{p_c}(0,v))}{\m(p_c-p)+\chi(p_c)^{-1}}.
	}
In particular, \refeq{chi(p)-asy} implies that $\chi(p)=\Theta(V^{1/3})$ whenever
$p\leq p_c$ is in the scaling window, i.e., when $\m(p_c-p)=O(V^{-1/3})$
as stated in \eqref{chi-asy-pc} in Theorem \ref{thm-crit}. Equation \refeq{chi(p)-asy}
also proves \eqref{chi-asy-sub} in Theorem \ref{thm-sub}.

We start by proving the upper bound in \refeq{chi(p)-asy}. \ch{We remark that this bound is valid for {\em all} transitive graphs (that is, we do not require here the triangle condition).} By Russo's formula,
	\eqn{
	\frac{d}{dp}\prob_p(0\conn x)=\sum_{(u,v)\in \Edges}
	\prob_p\big((u,v) \text{ is pivotal for }0\conn x\big),
	}
where we say that a (directed) bond $(u,v)$ is \emph{pivotal} for $0 \conn x$ when (a) $0\conn u$ and
(b) $0\conn x$ in the
(possibly modified) configuration where the status of $(u,v)$ is turned to occupied, while
$0$ is not connected to $x$ in the (possibly modified) configuration where the status of $(u,v)$
is turned to vacant. Summing over $x$ yields
    \eqn{
    \lbeq{der-chi-expr}
    \frac{\mathrm{d}}{\mathrm{d}p} \chi(p) = \sum_{x\in \Vertices}
    \sum_{(u,v)\in \Edges} \prob_p\big((u,v) \text{ is pivotal for }
    0\conn x\big).
    }

If $(u,v)$ is pivotal for $0\conn x$, then
there exist two disjoint paths of occupied bonds connecting
$0$ and $u$, and $v$ and $x$, respectively. Thus,
$\{0\conn u\}\circ \{v\conn x\}$ occurs. \ch{The} BK inequality \cite{Grim99} gives
    \eqn{
    \lbeq{der-bd}
    \frac{\mathrm{d}}{\mathrm{d}p} \chi(p) \leq \sum_{x\in \Vertices}
    \sum_{(u,v)\in \Edges} \prob_p(0\conn u)\prob_p(v\conn x)
    =\m \chi(p)^2.
    }
We rewrite the last inequality as $\frac{\mathrm{d}}{\mathrm{d}p} \chi(p)^{-1} \geq -\m$, and integrate over
$[p,p_c]$ to get
    \eqn{
    \lbeq{integrate}
    \chi(p_c)^{-1}-\chi(p)^{-1}\geq -\m(p_c-p),
    }
so that
    \eqn{
    \chi(p)\geq \frac{1}{\m(p_c-p)+\chi(p_c)^{-1}} \, ,
    }
showing the upper bound in \refeq{chi(p)-asy}. For the lower bound we write
    \eqn{
    \lbeq{lace-1}
    \prob_p((u,v) \text{ is pivotal for }
    0\conn x)
    =\expec_p[\indic{0\conn u}\tau^{\tilde \C^{(u,v)}(0)}(v,x)],
    }
where $\tilde \C^{(u,v)}(0)$ consists of those sites which are connected
to 0 without using the bond $(u,v)$ and for a set of sites $A$, the restricted two-point function
$\tau^{A}(v,x)$ is the probability that $v$ is connected to $x$ and every open path from $v$ to $x$ has all its edges not touching $A$. Clearly,
$\tau^{\tilde \C^{(u,v)}(0)}(v,x)\leq \tau(v,x)$, and
this is in fact an easy way to prove BK inequality for this particular events. We note that
    \eqn{
    \lbeq{lace-2}
    \prob_p(v\conn x)-\tau^{A}(v,x)=\prob_p(v\ct{A} x),
    }
where we write that $v\ct{A} x$ when every open path from $v$ to $x$ has an edge touching $A$. Thus,
    \eqan{
    \frac{\mathrm{d}}{\mathrm{d}p} \chi(p) &= \sum_{x\in \Vertices}
    \sum_{(u,v)\in \Edges}\expec_p[\indic{0\conn u}] \prob_p(v \lrfill x)
    \nonumber\\
    &\qquad-
    \sum_{x\in \Vertices}
    \sum_{(u,v)\in \Edges}\expec_p[\indic{0\conn u}\prob_p(v\ct{\tilde \C^{(u,v)}(0)} x)]\nonumber\\
    &=\m \chi(p)^2-
    \sum_{x\in \Vertices}
    \sum_{(u,v)\in \Edges}\expec_p[\indic{0\conn u}\prob_p(v\ct{\tilde \C^{(u,v)}(0)} x)].
    }
Now, for any $A\subseteq \Z^d$,
    \eqn{
    \lbeq{through-bd}
    \prob_p(v\ct{A} x)
    \leq \sum_{a} \prob_p\big(\{v\conn a\}\circ \{a\conn x\}\big)\indic{a\in A},
    }
which by BK inequality and summing over $x$ leads to
    \eqan{
    \frac{\mathrm{d}}{\mathrm{d}p} \chi(p)
    &\geq \m \chi(p)^2-
    \chi(p)
    \sum_{(u,v)\in \Edges}\sum_{a} \prob_p(0\conn u, a\in \tilde \C^{(u,v)}(0)) \prob_p(v\conn a)\\
	&\geq \m \chi(p)^2-
    \chi(p)
    \sum_{(u,v)\in \Edges}\sum_{a} \prob_p(0\conn u, 0 \conn a)\prob_p(v\conn a).\nn
    }
If $0 \conn u$ and $0 \conn a$, then there exists $z$ such that $\{0 \conn z\}\circ\{z \conn u\}\circ\{z \conn z\}$, so by the BK inequality
   \eqan{
    \frac{\mathrm{d}}{\mathrm{d}p} \chi(p)
    &\geq \m \chi(p)^2-
    \chi(p) \sum_{(u,v)\in \Edges}\sum_{a,z} \prob_p(0\conn z)\prob_p(z\conn u)\prob_p(z\conn a)\prob_p(v\conn a) \, . \nn
    }
The sum over $a,z$ looks almost like the triangle diagram, except for the the annoying $\prob_p(0\conn z)$ factor. However, by transitivity the double sum on the right hand side remains the same if we replace $0$ by any other vertex. Hence we may sum this over $0$, getting a factor of $\chi(p)$, and add a factor of $V^{-1}$. This gives that
$$\frac{\mathrm{d}}{\mathrm{d}p} \chi(p) \geq \m \chi(p)^2 - \m \chi(p)^2 \sum_{(0,v)\in \Edges} \nabla_p(0,v)\leq \m \chi(p) \max_{v\neq 0}\nabla_p(0,v) \, ,$$
implying that
    \eqn{
    \frac{\mathrm{d}}{\mathrm{d}p} \chi(p)\geq \m\chi(p)^2 [1-\max_{v\neq 0}\nabla_{p_c}(0,v)].
    }
We now integrate as we do in \refeq{integrate} and obtain the lower bound in \refeq{chi(p)-asy}.
\bigskip

\subsection{Random walk conditions for percolation} \label{sec-rwconditions}
We now describe a general theorem, obtained in \cite{HN12}, which allows to deduce as corollaries Theorems \ref{thm-sub}-\ref{thm-super} under certain geometric conditions on the underlying graphs. These conditions are more restrictive than the triangle condition, (for instance, they do not hold in the case of high-dimensional tori, but do hold for the hypercube) but are easier to verify since they are expressed in terms of random walks. In particular, these conditions imply the strong triangle condition (and hence by Theorem \ref{thm-BCHSS1} they imply Theorems \ref{thm-sub} and \ref{thm-crit}), but more importantly they allow us to analyze percolation in the supercritical case, where the triangle condition ceases to hold, and obtain Theorem \ref{thm-super}.

Let $G$ be a finite transitive graph on $V$ vertices and degree $\m$. Consider the non-backtracking random walk (NBW) on it (this is just a simple random walk not allowed to traverse back on the edge it just came from). For any two vertices $x,y$, we put $\p^t(x,y)$ for the probability that the walk started at $x$ visits $y$ at time $t$. We write $\Tm$ for the {\em uniform mixing time} of the walk, that is,
    \eqn{
    \Tm = \min \Big \{ t \colon  \max_{x,y} \,\,{\p^t(x,y) + \p^{t+1}(x,y) \over 2}
    \leq (1+o(1)) V^{-1} \Big \} \, ,
    }
where $o(1)$ tends to $0$ slowly. Then the main result in \cite{HN12} is as follows:

\begin{theorem}[\cite{HN12}]
\label{mainthmgeneral}
Let $G$ be a transitive graph on $V$ vertices with degree $\m$ and define $p_c$ as in (\ref{pcdef})
with $\lambda=1/10$. Assume that the following conditions hold:
\begin{enumerate}
\item $\m \to \infty$,
\item $[p_c(\m-1)]^{\Tm} = 1 + o(1)$,
\item For any vertices $x,y$,
    \eqn{
    \sum_{u,v} \sum_{\substack{t_1,t_2,t_3=0\\t_1+t_2+t_3\geq 3}}^{\Tm} \p^{t_1}(x,u) \p^{t_2}(u,v) \p^{t_3}(v,y) =o(1/\log{V}).
    }
\end{enumerate}
Then,
\begin{enumerate}
\item[(a)] the finite triangle condition (\ref{finitetriangle}) holds (and hence the assertions of
Theorems \ref{thm-sub}-\ref{thm-crit} hold),
\item[(b)] for any sequence $\eps=\epsm$ satisfying $\epsm \gg V^{-1/3}$ and $\epsm = o(\Tm^{-1})$,
    \eqn{
    \frac{|\C_1|}{2\epsm V}\convp 1 \,  , \quad\qquad \E_p |\C(0)| = (4+o(1))\epsm^2 V \,  , \quad\qquad \frac{|\C_2|}{\epsm V}\convp 0 \, .
    }
\end{enumerate}
\end{theorem}

In Section \ref{sec-overview} we will part (a) of the Theorem above, and in Section \ref{sec-unlacing} we will verify the conditions of the Theorem. Hence, we will obtain a proof that the triangle condition holds on the hypercube. Note that condition (2) involves both a random walk estimate (bounding $\Tm$) and a percolation estimate (bounding $p_c$). Let us now discuss how the verification of these conditions is done in a rather elementary way.

\subsection{Back to the hypercube}\label{sec-backtohypercube} It is a classical fact that the total-variation mixing time of the random walk on the hypercube is of order $\m \log{\m}$ \cite{LPW09}. A separate argument is needed to show that this is the correct order for $\Tm$ since (a) we are dealing with the non-backtracking walk; and (b) we require a bound on the stronger uniform mixing time. This can be done by analyzing the transition matrix of the non-backtracking random walk using classical tools. This analysis is performed by the Fitzner and the first author in \cite{FitHof11a} and also allows us to verify condition (3) of Theorem \ref{mainthmgeneral}. We will not delve further into this part of the proof.

Thus, given that $\Tm = \m \log{\m}$, the verification of condition (2) in Theorem \ref{mainthmgeneral} in the case of the hypercube simply states that $p_c = {1 \over m-1} + o(m^{-2} \log m)$. This estimate (and more) was already proved in the work of the first author and Slade \cite{HS05, HS06} so no further estimates on $p_c$ were required in \cite{HN12}. However, the estimate we require is much weaker and the proofs in \cite{HS05, HS06} are difficult and rely on the lace expansion. In this paper we provide an elementary argument giving this estimate. This is the last piece in the ``unlacing'' puzzle which verifies condition (2) of Theorem \ref{mainthmgeneral} in the case of the hypercube.

\begin{theorem}[Unlacing the lace expansion in the hypercube]
\label{thm-unlacing}
Consider bond percolation on the hypercube $\{0,1\}^{\m}$. Then, there exists $C>0$ such that
	\eqn{
	p_c\leq {1+{5/(2\m^2)}+{C/\m^{-3}} \over \m-1} \, .
	}
Consequently, $[(\m-1)p_c]^{\Tm}=1+o(1)$, so that the results in
Theorems \ref{thm-BCHSS1} and \ref{mainthmgeneral} apply
(and hence also the assertions of Theorems \ref{thm-sub}-\ref{thm-super}).
\end{theorem}

\noindent This theorem is the only novel result of this paper, and
is proved in Section \ref{sec-unlacing}. Let us further briefly discuss the precise value of $p_c$ since it is related to the early literature on hypercube percolation.
\medskip

\paragraph{{\bf The asymptotic expansion of $p_c$}}
The problem of establishing a phase transition for the appearance of a component of size of order $2^{\m}$ was solved in the breakthrough work of Ajtai, Koml\'{o}s and Szemer\'{e}di \cite{AKS82}. They proved that when the retention probability of an edge is scaled as $p=c/\m$ for a fixed constant $c>0$ the model exhibits a phase transition: for $c<1$, the largest component has size of order $\m$ whp, while for $c>1$, the largest component has size linear in $2^\m$ whp. See also \cite{FilPem93, Penrose98} for proofs that the giant component has size $\zeta(c)2^{\m}(1+o(1))$ whp when $p=c/\m$, where $\zeta(c)$ is the survival probability of a Poisson branching process with expected offspring equal to $c$. Thus, $p_c\approx 1/\m$, but it was unclear at that time just how close it is.

The first improvement to \cite{AKS82} was obtained by Bollob\'as, Kohayakawa and \L uczak \cite{BolKohLuc92}. They showed that if $p=(1+\epsm)/\m$ with $\epsm=o(1)$ but $\epsm \geq 60\m^{-1} (\log \m)^3$, then $|\C_1|=(2+o(1))\epsm 2^\m$ \whp. This raises the question whether $p_c=1/(\m-1)$, which is answered negatively in \cite{HS05, HS06}. These results give the most precise estimates on $p_c$ to date:


\begin{theorem}[Asymptotic expansion of $p_c$ \cite{HS05, HS06}]
\label{thm-asy-exp}
For bond percolation on the hypercube $\{0,1\}^{\m}$, there exist rational coefficients $(a_i)_{i\geq 1}$
with $a_1=a_2=1, a_3=7/2$ such that, for every $s\geq 1$, as $\m\rightarrow \infty$,
	\eqn{
	\label{pc-asyp-exp}
	p_c=\sum_{i=1}^s a_i \m^{-i} +O(\m^{-(s+1)}).
	}
\end{theorem}



Note that by Theorem \ref{thm-sub}, \ref{thm-super} and \ref{thm-asy-exp}, whatever $s\geq 2$ is, the
largest cluster jumps from $O(\m^{2s-1})$ for $p=\sum_{i=1}^s a_i \m^{-i}+\eta \m^{-s}$
with $\eta<0$ to $\Theta(2^{\m}/\m^{s-1})$ for $p=\sum_{i=1}^s a_i \m^{-i}+\eta \m^{-s}$
with $\eta>0$. Thus, the phase transition in $\eta$ is extremely sharp for every $s\geq 2$ fixed.



\subsection{Open problems} \label{sec-open}
We here collect a list of what we consider to be important open problems in this area.
Some of these problems appear in \cite[Section 8]{HN12}, but not all.

\begin{enumerate}

\item[1.] {\bf Percolation on high-dimensional tori.} Consider bond percolation on the nearest-neighbor torus $\Z_n^d$ where $d$ is a large fixed constant and $n\to \infty$ with $p=p_c(1+\epsn)$ such that $\epsn \gg n^{-d/3}$ and $\epsn = o(1)$. Show that $|\C_1|/(\epsn n^d)$ converges to a constant. Does this constant equal the limit as $\vep\downarrow 0$ of $\eps^{-1} \theta_{\Z^d} (p_c(1+\eps))$? Here $\theta_{\Z^d}(p)$ denotes the probability that the cluster of the origin is infinite at $p$-bond percolation on the infinite lattice $\Z^d$.


\item[2.] {\bf Identify the scaling limit of cluster sizes in the scaling window.} Show that $(|\C_j|2^{-2m/3})_{j\geq 1}$ converges in distribution when $p=p_c(1+t 2^{-\m/3})$ and $t \in  {\mathbb R}$ is fixed and identify the limit distribution. This should be the limiting distribution of critical clusters in $G(n,p)$ as identified by Aldous \cite{Aldo97}.

    We remark that in \cite{HH2} it is proved that any subsequential limit of $\{|\C_1|2^{-2\m/3}\}_{\m\geq 1}$ is
a proper random variable, that is, when $\prob(X=\E[X])<1$. This {\em non-concentration} is the hallmark of critical behavior.

\item[3.] {\bf Prove that the discrete duality principle holds for hypercube percolation.}
Show that $|\C_1| = (2+o(1)) \eps^{-2} \log(\eps^3 2^m)$ when $\epsm \ll -2^{-\m/3}$ and $\epsm = o(1)$ and that $|\C_2| = (2+o(1)) \eps^{-2} \log(\epsm^3 2^{\m})$ when $p=p_c(1+\epsm)$ with $\epsm\gg 2^{-\m/3}$.
This is also the content of \cite[Conjectures 3.1 and 3.3]{BCHSS3} and is proved for some values of $\epsm$ in \cite{BolKohLuc92}. In $G(n,p)$ these results are proved in
 \cite{PitWor05} and \cite[Theorem 5.6]{JLR00}.

\item[4.] {\bf Prove a central limit theorem for $|\C_1|$.}
Show that $|\C_1|$ satisfies a central limit theorem throughout the supercritical
regime. In $G(n,p)$ this and much more was established by Pittel and Wormald \cite{PitWor05}.


\item[5.] {\bf Unlace the asymptotic expansion of $p_c$.} Find a proof of Theorem \ref{thm-asy-exp} that does not rely on the lace expansion. Possibly, the ideas in the proof of Theorem \ref{thm-unlacing} in Section \ref{sec-unlacing}
can be used.

\item[6.] {\bf Compute further coefficients of the asymptotic expansion of $p_c$.} Find the numerical values of $a_i$ for $i\geq 4$ in Theorem \ref{thm-asy-exp}. There is a large physics literature on asymptotic
expansions of critical values. See e.g., \cite{HS05} for some of the references.
We expect that $a_4=16$, as the first 4 coefficients of
the asymptotic expansion of $p_c$ can be expected to be the same as the ones for the asymptotic expansion of $p_c({\mathbb Z}^d)$ in terms of inverse powers of $2d$ (see e.g.\ \cite{GauRus78}).
We also expect that $a_5$ is \emph{not}
equal to the 5th coefficient in the asymptotic expansion of $p_c({\mathbb Z}^d)$ in terms of
$1/(2d)$, which is predicted to be equal to 103 \cite{GauRus78}. Recently,
substantial progress was made for the asymptotic
expansion of the connective constant for self-avoiding walk on ${\mathbb Z}^d$, for which
the first 13 coefficients have been computed by Clisby, Liang and Slade
(see \cite{CliLiaSla07, CliSla09}).
\end{enumerate}



\subsection{Acknowledgements}
The work of RvdH was supported in part by the Netherlands
Organization for Scientific Research (NWO). The work of AN was partially supported by NSF and NSERC grants.
This work was presented by RvdH on the occasion of the Stochastik Tage 2012, held in Mainz
March 6-9, 2012.

\section{Overview of the proof of the supercritical phase}
\label{sec-overview}

In this section we give an overview of the key steps in the proofs in \cite{HN12}.
From here on, we assume that $\epsm$ is a sequence
such that $\epsm=o(1)$ but $\epsm^3 V\to \infty$.

\subsection{Notations and tools}
\label{sec-notation}
We write $d_{G_p}(x,y)$ for the length of a shortest
$p$-open path between $x,y$ and put $d_{G_p}(x,y)=\infty$ if $x$ is not connected
to $y$ in $G_p$. We write $x \lrr y$ if $d_{G_p}(x,y) \leq r$ and
$x \stackrel{=r}{\lrfill} y$ if $d_{G_p}(x,y)=r$ and $x \stackrel{[a,b]}{\lrfill} y$
if $d_{G_p}(x,y) \in [a,b]$.
The {\em intrinsic metric} ball of radius $r$ around $x$ and its boundary are defined by
    \eqn{
    B_x(r) = \{ y \colon d_{G_p}(x,y) \leq r\} \, , \qquad
    \partial B_x(r) = \{ y \colon d_{G_p}(x,y)=r\} \, .
    }
Note that these are random sets of the graph and not the balls in shortest path metric
of the graph $G$.
We often drop $0$ from notation and write $B(r)$ for $B_0(r)$ whenever possible.

\subsection{Tails of the supercritical cluster size}
\label{sec-super-clus-tail-over}
We start by describing the tail of the cluster size in the supercritical regime.

\begin{theorem}[Bounds on the cluster tail]
\label{clustertail}
Let $G$ be a finite transitive graph of degree $\m$ on $V$ vertices such that the finite
triangle condition (\ref{finitetriangle}) holds and put $p=p_c(1+\epsm)$
where $\epsm = o(1)$ and $\epsm \gg V^{-1/3}$.
Then, for the sequence $k_0=\epsm^{-2} (\epsm^3 V)^{1/4}$,
    \be
    \label{lb-cluster-tail}
    \prob ( |\C(0)| \geq k_0 ) = 2 \epsm (1+o(1)) \, .
    \ee
\end{theorem}

This Theorem is reminiscent of the fact that a branching process
with Poisson progeny  distribution of mean $1+\eps$ has survival
probability of $2\eps(1+O(\vep))$. Upper and lower
bounds of order $\eps$ for the cluster tail were proved already in
\cite{BCHSS2} using Barsky and Aizenman's differential inequalities \cite{BA91},
and were sharpened in \cite[Appendix A]{HN12} to obtain the right constant 2.

Let $Z_{\sss \geq k}$ denote the number of vertices with cluster size at least $k$, i.e.,
    \be
    Z_{\sss \geq k} = \big | \big \{ v \colon |\C(v)| \geq k\big \} \big | \, .
    \ee
We use Theorem \ref{clustertail} to show that $Z_{\sss \geq k_0}$, with $k_0$ as
in the theorem, is concentrated.

\begin{lemma}[Concentration of $Z_{\sss \geq k_0}$]
\label{lem-conc-Z-I}
In setting of Theorem \ref{clustertail}, if $m\to \infty$, then
    \eqn{
    \frac{Z_{\sss \geq k_0}}{2\eps V}\convp 1,
    \qquad \text{and}
    \qquad
    \E |\C(0)| \leq (4+o(1))\eps^2 V \, .
    }
\end{lemma}
Lemma \ref{lem-conc-Z-I} immediately proves the upper bound on
$|\C_1|$ in Theorem \ref{thm-super}:
\medskip
\paragraph{{\bf Proof of upper bound on $|\C_1|$ in Theorem \ref{thm-super}.}}
Note that $\{|\C_1|\geq k\}=\{Z_{\sss \geq k}\geq 1\}$, so that
$|\C_1|\leq Z_{\sss \geq k}$ on the event $\{Z_{\sss \geq k}\geq 1\}$.
Applying this to $k=k_0$ and using Lemma \ref{lem-conc-Z-I} proves the upper bound in
Theorem \ref{thm-super}.
\qed

\subsection{Uniform connection bounds and the role of the random walk}
\label{sec-unif-conn-NBW}
We expand here on one of our most useful estimates on percolation connection probabilities. In its proof, a simple key connection between percolation and the mixing time of the non-backtracking walk is revealed. In the analysis of the Erd\H{o}s-R\'enyi random graph $G(n,p)$ symmetry plays a special role.
One instance of this symmetry is that the function
$f(x) = \prob(0 \lr x)$ is constant whenever $x \neq 0$ and its value is
precisely $(V-1)^{-1}(\E|\C(0)|-1)$ and $1$ when $x=0$. Such a statement
clearly does not hold on the hypercube at $p_c$: the probability that
two neighbors are connected is at least $p_c\geq \m^{-1}$, while the probability that
$0$ is connected to one of the vertices in the barycenter of the cube is
at most $\sqrt{\m} 2^{-\m} \E|\C(0)|$ by symmetry.

A key observation in the proof of Theorem \ref{thm-super} in \cite{HN12}
is that one can recover this symmetry as long as we
require the connecting paths to be longer than the mixing time of the random walk,
as shown in \cite[Lemma 3.12]{HN12}:

\begin{lemma}[Uniform connection estimates]
\label{moreunifconnbd}
Perform bond percolation on any graph $G$ satisfying
the assumptions of Theorem \ref{mainthmgeneral}. Then, for every $r\geq \mnot$ and any vertex $x\in G$
    \be
    \label{exampleunif}
    \prob_{p_c} (0 \stackrel{[\mnot, r]}{\lrfill} x) \leq (1+ o(1)){\E |B(r)|\over V}\, ,
    \ee
where $\mnot$ is uniform mixing time as defined above Theorem \ref{mainthmgeneral}.
In particular,
    \be
    \label{exampleunifb}
    \prob_{p_c} (0 \stackrel{[\mnot, \infty)}{\lrfill} x) \leq (1+ o(1)){\E |\C(0)|\over V}\, .
    \ee
\end{lemma}


The proof of the above lemma is short and elementary, see \cite{HN12}. There it is also shown how to obtain similar estimates for $p=p_c(1+\eps)$ (with an error depending on $\eps$). The uniformity of this lemma allows us to decouple the sum in the triangle diagram and yields a simple proof of the strong triangle condition, as we now show. \\

\noindent {\bf Proof of part (a) of Theorem \ref{mainthmgeneral}}. Let $p \leq p_c$. If one of the connections in the sum $\nabla_{p}(x,y)$ is of length in $[\mnot, \infty)$, say between
$x$ and $u$, then we may estimate
    \eqan{
    \sum _{u,v} \prob_{p}(x\stackrel{[\mnot, \infty)}{\lrfill} u) \prob_{p}(u \lr v) \prob_{p}(v \lr y)
     &\leq {(1+o(1)) \E_{p} |\C(0)| \over V} \sum_{u,v} \prob_{p}(u \lr v) \prob_{p}(v \lr y)
     \\ &= {(1+o(1)) (\E_{p} |\C(0)|)^3 \over V} \, ,\nn
    }
where we have used Lemma \ref{moreunifconnbd} for the first inequality. Thus, we are only left to deal with short connections:
    \eqn{
    \nabla_{p}(x,y) \leq \sum_{u,v} \prob_{p}(x \stackrel{\mnot}{\lrfill} u)
    \prob_{p}(u \stackrel{\mnot}{\lrfill} v)\prob_{p}(v \stackrel{\mnot}{\lrfill} y) + O(\chi(p)^3/V) \, .
    }
We write
    \eqn{
    \prob_{p}(x \stackrel{\mnot}{\lrfill} u) = \sum_{t_1=0}^{\mnot} \prob_{p}(x \stackrel{=t_1}{\lrfill} u) \, ,
    }
and do the same for all three terms so that
    	\be\label{triangle.midstep} \nabla_{p}(x,y) \leq
	\sum_{u,v} \sum_{t_1,t_2,t_3}^{\mnot} \prob_{p}(x \stackrel{=t_1}{\lrfill} u)
	\prob_{p}(u \stackrel{=t_2}{\lrfill} v) \prob_{p}(v \stackrel{=t_3}{\lrfill} y) + O(\chi(p)^3/V) \, .
    	\ee
We bound
    	\eqn{
	\prob_{p}(x \stackrel{=t_1}{\lrfill} u) \leq \m(\m-1)^{t_1-1} \p^{t_1}(x,u) p^{t_1} \, ,
    	}
simply because $\m(\m-1)^{t_1-1} \p^{t_1}(x,u)$ is an upper bound on the number of
simple paths of length $t_1$ starting at $x$ and ending at $u$. Hence
    	\eqn{
	\nabla_{p}(x,y) \leq {\m^3 \over (\m-1)^3} \sum_{u,v}
	\sum_{\substack{t_1,t_2,t_3}}^{\mnot} [p(m-1)]^{t_1+t_2+t_3} \p^{t_1}(x,u) \p^{t_2}(u,v) \p^{t_3}(v,y)
	+ O(\chi(p)^3/V) \, .
	}
Since $p \leq p_c$, assumption (2) gives that $[p(m-1)]^{t_1+t_2+t_3} = 1 + o(1)$, and
it is a simple consequence of condition (3) that
	\eqn{
	\sum_{u,v}
	\sum_{\substack{t_1,t_2,t_3}}^{\mnot} [p(m-1)]^{t_1+t_2+t_3} \p^{t_1}(x,u) \p^{t_2}(u,v) \p^{t_3}(v,y)
	\leq {\bf 1}_{\{x=y\}} + o(1) \, ,
	}
where $o(1)$ vanishes as $\m\rightarrow \infty$, concluding the proof. \qed

\subsection{Most large cluster share large boundary}
\label{sec-large-boundary}

Since this is the most technical part of the overview, at the expense of being precise, we
have chosen to reduce the clutter of notation and suppress several parameters from the
notation. We ignore several dependencies between parameters and the skeptical reader is
welcomed to read the more precise overview presented in \cite{HN12}.

Two parameters however play an important role. We choose $r$ and $r_0$ so that $r \gg \epsm^{-1}$ but just barely, and $r_0 \gg r$ in a way that will become clear later. For vertices $x,y$, define the random variable
    	\eqn{
    	S_{r+r_0}(x,y) =
	\big | \big \{ (u,u')\in E(G) \colon \{x \stackrel{r+r_0}{\lrfill} u\} \circ \{y \stackrel{r+r_0}{\lrfill} u'\} \, ,
	|B_u(r+r_0)|\cdot |B_{u'}(r+r_0)| \leq \eps^{-2} (\E |B(r_0)|)^2 \big \} \big | \, .
    	}
The edges counted in $S_{r+r_0}(x,y)$ are the ones that we are going to sprinkle.
Informally, a pair of vertices $(x,y)$ is \emph{good} when their clusters
are large and $S_{r+r_0}(x,y)$ is large, so that their clusters have many edges
between them. We make this quantitative in the following definition:

\bde[$(r,r_0)$-good pairs]\label{goodpairs}
We say that $x,y$ are $(r,r_0)$-good if all of the following occur:
\begin{enumerate}
\item $\partial B_x(r) \neq \emptyset$, $\partial B_y(r)\neq \emptyset$
 and $B_x(r) \cap B_y(r) = \emptyset$,
\item $|\C(x)|\geq (\epsm^3 V)^{1/4} \epsm^{-2}$ and $|\C(y)| \geq (\epsm^3 V)^{1/4} \epsm^{-2}$,
\item $S_{2r+r_0}(x,y) \geq V^{-1} \m \epsm^{-2} (\E|B(r_0)|)^2$.
\end{enumerate}
Write $\Park$ for the number of $(r,r_0)$-good pairs.
\ede
%

\begin{theorem} [Most large clusters share many boundary edges]
\label{alphapairs}
Let $G$ be a graph on $V$ vertices and degree $\m$ satisfying the assumptions in Theorem \ref{mainthmgeneral}. Assume that $\epsm$ satisfies $\epsm \gg V^{-1/3}$ and $\epsm = o(\mnot^{-1})$. Then,
    $$
    \frac{\Park}{(2\epsm V)^2}\convp 1\, .
    $$
\end{theorem}

In light of Theorem \ref{clustertail}, we expect that the number of pairs
of vertices $(x,y)$ with $|\C(x)|\geq (\epsm^3 V)^{1/4}\epsm^{-2}$ and
$|\C(y)| \geq (\epsm^3 V)^{1/4}\epsm^{-2}$
is close to $(2\epsm V)^2$. Theorem \ref{alphapairs} shows that almost all of these pairs
have clusters that share many edges between them.
Theorem \ref{alphapairs} allows us to prove
Theorem \ref{mainthmgeneral}, as we describe in more detail in the next section.

The difficulty in Theorem \ref{alphapairs}
is the requirement (3) in Definition \ref{goodpairs}. Indeed, conditioned on survival (that is,
on $\partial B_x(r) \neq \emptyset$, $\partial B_y(r)\neq \emptyset$ and that the balls are disjoint),
the random variable $S_{r+r_0}(x,y)$ is not concentrated and
hence it is hard to prove that it is large. In fact, even the variable
$|B(r_0)|$ is not concentrated. This is not a surprising fact: the number of descendants at
generation $n$ of a branching process with mean $\mu>1$ divided by $\mu^n$ converges as
$n\to \infty$ to a non-trivial random variable. Non-concentration
occurs because the first generations of the process have a strong and lasting effect
on the future of the population. In \cite{HN12}, we counteract this non-concentration by
conditioning on the whole structure of $B_x(r)$ and $B_y(r)$. Since $r$ is bigger than the correlation length ($r \gg \epsm^{-1}$), under this conditioning the variable $S_{r+r_0}(x,y)$ is concentrated (as one would expect from the branching process analogy).

\subsection{Sprinkling and improved sprinkling}
\label{sec-intsprinkling}
The sprinkling technique was invented by Ajtai, Koml\'{o}s and Szemer\'{e}di
\cite{AKS82} to show that $|\C_1|=\Theta(2^\m)$ when $p=(1 +\eps)/\m$ for fixed $\eps>0$
and can be described as follows. Fix some small $\theta>0$ and write
$p_1 = {(1 + (1-\theta)\eps)/\m}$ and $p_2\geq \theta \eps/ \m$
such that $(1-p_1)(1-p_2)=1-p$. It is clear that $G_p$ is distributed
as the union of the edges in two independent copies of $G_{p_1}$
and $G_{p_2}$. The sprinkling method consists
of two steps. The first step is performed in $G_{p_1}$ and uses a branching
process comparison argument together with an Azuma-Hoeffding concentration inequality to obtain
that \whp\ at least $c_2 2^{\m}$ vertices are contained in connected
components of size at least $2^{c_1 \m}$ for some small but fixed
constants $c_1, c_2>0$. In the second step we
add the edges of $G_{p_2}$ (these are the ``sprinkled'' edges) and show that they
connect many of the clusters of size at least $2^{c_1 \m}$ into a giant cluster
of size $\Theta(2^\m)$.

Let us give some details on how the last step is done. A key tool here is the
\emph{isoperimetric inequality} for the hypercube stating that two disjoint
subsets of the hypercube of size at least $c_2 2^{\m}/3$ have at least $2^\m/\m^{100}$ disjoint
paths of length $C(c_2)\sqrt{\m}$ connecting them, for some constant $C(c_2)>0$.
(The $\m^{100}$ in the denominator
is not sharp, but this is immaterial as long as it is a polynomial in $\m$.)
This fact is used in the following way. Write $V'$ for the set of vertices which
are contained in a component of size at least $2^{c_1 \m}$ in $G_{p_1}$ so that
$V'\geq c_2 2^\m$. We say that \emph{sprinkling fails} when
$|\C_1| \leq c_2 2^\m/3$ in the union $G_{p_1} \cup G_{p_2}$. If
sprinkling fails, then we can partition $V'=A \uplus B$ such that both
$A$ and $B$ have cardinality at least $c_2 2^\m/3$ and {\em any}  path of length
at most $C(c_2)\sqrt{\m}$ between them has an edge which is $p_2$-closed. The number
of such partitions is at most $2^{2^\m/2^{c_1 \m}}$. The probability that a path
of length $k$ has a $p_2$-closed edge is $1- p_2^k$. Applying the isoperimetric inequality
and using that the paths guaranteed to exist by it are disjoint so that
the edges in them are independent, the probability that
sprinkling fails is at most
    \be
    \label{easysprinkle}
    2^{2^\m/2^{c_1 \m}} \cdot
    \Big (1 - \big ( {\theta\eps \over \m} \big )^{C(c_2)\sqrt{\m}} \Big )^{2^\m/\m^{100}}
    = \e^{-2^{(1+o(1))\m}}\, ,
    \ee
which tends to 0.


The sprinkling argument above is not optimal due to the use of the isoperimetric
inequality. It is wasteful because it assumes that large percolation clusters
can be ``worst-case'' sets, that is, sets which saturate the isoperimetric
inequality (e.g., two balls of radius $\m/2-\sqrt{\m}$ around two vertices
at Hamming distance $\m$).
However, it is in fact very improbable for percolation clusters to be similar to this kind of worst-case sets. In \cite{HN12}, this is replaced by an argument showing that percolation clusters are ``close'' to uniform random sets of similar size, so that two large clusters share many closed edges with the property that if
we open even \emph{one} of them, then the two clusters connect.

\medskip
\paragraph{{\bf Improved sprinkling: Proof of Theorem \ref{mainthmgeneral}(b).}}
Recall that we already proved the upper bound on $|\C_1|$ below Lemma \ref{lem-conc-Z-I}, so it remains to show that
    \eqn{
    \label{mainthm.toshow}
    \prob_p \big( |\C_1| \geq (2-o(1)) \epsm V\big) = 1-o(1) \, .
    }
Recall that $p=p_c(1+\epsm)$ is our percolation probability and choose $p_1, p_2$ satisfying
    \eqn{
    p_2 = \theta \epsm /m \, , \qquad p_c(1+\epsm) = p_1 + (1-p_1)p_2 \, ,
    }
where $\theta>0$ tends to $0$ extremely slowly so that $p_1=[1+(1-o(1))\epsm)]p_c$. Denote
by $G_{p_1}$ and $G_{p_2}$ as before.
We first invoke Theorem \ref{alphapairs} in $G_{p_1}$ and deduce
that \whp
    \be
    \label{mainthm.alphapairsapp}
    \Park = (1-o(1)) 4\epsm^2 V^2 \, .
    \ee
Now we wish to show that when we ``sprinkle'' this configuration in $G_{p_1}$,
that is, when we add to the configuration independent $p_2$-open edges,
most of these vertices join together to form one cluster of size roughly $2 \epsm V$.
We construct an auxiliary simple graph $H$ with vertex set
    $$
    V(H) = \big \{ x\in G_{p_1} \colon |\C(x)|\geq (\epsm^3 V)^{1/4} \epsm^{-2} \big \} \, ,
    $$
and edge set
    $$
    E(H) = \big \{ (x,y) \in V(H)^2 \colon x,y \hbox{ are $(r,r_0)$-good} \big \} \, .
    $$
Lemma \ref{lem-conc-Z-I} and (\ref{mainthm.alphapairsapp}) now imply that \whp\ $H$ is almost the complete graph, that is
    \eqn{
	\label{auxgraph}
        |V(H)| =  (2+o(1)) \epsm V \, , \quad\qquad |E(H)| = (1-o(1)) 4 \epsm^2 V^2 \, .
    }

Denote $v=|V(H)|$ so that $v=(2+o(1)) \epsm V$ and write $x_1, \ldots, x_v$ for the vertices in $G_{p_1}$ corresponding
to those of $H$. Given $G_{p_1}$ for which the event in (\ref{auxgraph}) occurs,
we will show that \whp\ in $G_{p_1} \cup G_{p_2}$ there is no way to partition
the set of vertices into  $M_1 \uplus M_2 = \{x_1,\ldots, x_v\}$ with
$|M_1| \geq \Omega(\epsm V)$ and $|M_2| \geq \Omega(\epsm V)$ such that there is no open path in
$G_{p_1} \cup G_{p_2}$ connecting a vertex in $M_1$ with a vertex in $M_2$.
This implies that \whp\ the largest connected component in $G_{p_1} \cup G_{p_2}$
is of size at least $(2-o(1))\epsm V$.

To show this, we first note that the number of
such partitions is at most $2^{3 (\epsm^3 V)^{3/4}}$
since $|\C(x_i)| \geq (\epsm^3 V)^{1/4} \epsm^{-2}$.
Secondly, given such a partition consisting of $M_1$ and $M_2$, we
claim that the number of edges $(u,u')\in E(H)$ such that $u \in M_1$ and $u'\in M_2$ (note that, by definition,
these edges must be $p_1$-closed) is at least $\Omega(\epsm^2 V \m)$.
To see this, we consider the set of edges in $H$ for which both sides lie in either
$M_1$ or $M_2$ (more precisely, the vertices of $H$ corresponding to $M_1$ and $M_2$). This number is clearly at most
    $$
    {M_1^2 + M_2^2} \leq (4-\Omega(1))\eps^2 V \, .
    $$
Hence, by \eqref{auxgraph}, the number of edges in $H$ such that one end is in $M_1$ and the other in $M_2$ is at least $\Omega(\epsm^2 V^2)$. In other words, there are at least $c \epsm^2 V^2$ pairs $(x,y)\in M_1\times M_2$ such that
$S_{r+r_0}(x,y) \geq  c V^{-1}\m \epsm^{-2} (\E|B(r_0)|)^2$. We choose $r_0$ so that this is a large number.
In total, we counted at least order
$\eps^2 V^2 \cdot V^{-1} \m \epsm^{-2} (\E|B(r_0)|)^2$ edges $(u,u')$ and no edge is counted more than
$|B_u(r+r_0)|\cdot|B_{u'}(r+r_0)|$ times, which is at most order $\epsm^{-2} (\E|B(r_0)|)^2$
by the definition of $S_{r+r_0}(x,y)$ and the second claim follows.

Hence if $|\C_1| \leq (2-\Omega(1))\eps V$ after the sprinkling, then there exists such a partition in which all of the
above edges $(u,u')$ are $p_2$-closed. By the two claims above, the probability of this is at most
    $$
    2^{3 (\epsm^3 V)^{3/4}} (1-p_2)^{c \epsm^2 V\m} = o(1) \, ,
    $$
since $p_2= \theta \epsm / \m$ and $\theta$ goes to $0$ very slowly. This establishes the required estimate on $|\C_1|$.

We now use \eqref{mainthm.toshow} to show the required bounds on $\expec|\C(0)|$ and $|\C_2|$.
The upper bound $\E|\C(0)|\leq (4+o(1))\epsm^2 V$ is stated in Lemma \ref{lem-conc-Z-I} and the
lower bound follows immediately from our estimate on $|\C_1|$, since
    $$
    \E |\C(0)| = V^{-1} \sum_{v \in V(G)} \E |\C(v)|
    = V^{-1} \sum_{j \geq 1} \E |\C_j|^2 \geq V^{-1} \E|\C_1|^2 \geq (4-o(1))\epsm^2 V \, ,
    $$
where the first equality is by transitivity, the second equality is because each
component $\C_j$ is counted $|\C_j|$ times in the sum on the left and the last
inequality is due to (\ref{mainthm.toshow}). Furthermore, by this inequality and
Lemma \ref{lem-conc-Z-I}, we deduce that
    $$
    \sum_{j \geq 2} \E|\C_j|^2 = o(\epsm^2 V^2) \, ,
    $$
and hence $|\C_2| = o(\epsm V)$ \whp. This concludes the proof of Theorem \ref{mainthmgeneral}.
\qed

\section{Unlacing hypercube percolation: Proof of Theorem \ref{thm-unlacing}}
\label{sec-unlacing}
The main result in this section is the following proposition:

\begin{proposition}[Expectation of intrinsic balls]
\label{prop-unlacing}
Consider bond percolation on the hypercube $\{0,1\}^\m$ with $p=[1+5/(2\m^2)+B/\m^3]/(\m-1)$
for some $B>0$ sufficiently large. Then, for $\m$ sufficiently large there exists a $k\geq 1$ such that
	\eqn{
	\E|B(k)|\geq 2^{\m/2}/\m^3.
	}
Consequently, $p_c\leq [1+5/(2\m^2)+B/\m^3]/(\m-1)=1/\m+1/\m^2+7/(2\m^3)+\Theta(1/\m^4)$.
\end{proposition}

The proof uses very elementary estimates on the non-backtracking
random walk transition probabilities. For completeness we provide here the
crude bounds that we will use, and remark that much more precise bounds are
available in \cite{FitHof11a}.

\begin{lemma}[NBW computations]
\label{nbw1}
Let $e_1=(1,0,\ldots,0) \in \{0,1\}^{\m}$ and $e_{1,1} =(1,1,0,\ldots,0)\in\{0,1\}^{\m}$
be hypercube vectors. Then
	$$
	\p^2(0,e_{1,1}) = {2 \over m (m-1)} \, ,
	$$
and for any fixed $t_0 \geq 2$ there exists $C=C(t_0)>0$ such that for all $t \geq t_0$,
	$$
	\p^{2t}(0,e_{1,1}) \leq C m^{-t_0-1} \, .
	$$
Furthermore,
	$$
	\p^3(0,e_1) = {1 \over m(m-1)} \, ,
	$$
and for any fixed $t_0 \geq 2$ there exists $C=C(t_0)>0$ such that for all $t \geq t_0$,
	$$
	\p^{2t+1}(0,e_1) \leq C m^{-t_0-1} \, .
	$$
\end{lemma}

\begin{proof}
The equality involving $\p^2(0,e_{1,1})$ is immediate since the probability that the non-backtracking walk takes any one of the two paths of length two from $0$ to $e_{1,1}$ is $[m(m-1)]^{-1}$. For the second inequality, denote by $X_t\in \{0,1\}^\m$ the location of non-backtracking random walk after $t$ steps and by $N_t$ the number of $1$'s in $X_t$. First note that by symmetry we have
	$$
	\prob(X_{2t} = e_{1,1} \, \mid \, X_0 = 0, N_{2t} = 2) = {2 \over m(m-1)} \, .
	$$
So let us estimate the probability that $N_{2t}=2$. This event implies that $N_{2t-2t_0+k} \leq 2t_0+2$ for any $1 \leq k \leq 2t_0$. Hence, on this event the process $\{N_{2t-2t_0+k}\}_{k=1}^{2t_0}$ is stochastically bounded below by a process $\{M_k\}_{k=1}^{2t_0}$ that has independent increments taking the value $1$ with probability $1-2t_0/\m$ and $-1$ with probability $2t_0/\m$ and $M_1$ satisfying $0 \leq M_1 \leq 2t_0$. If $N_{2t}=2$, then $M_{2t_0} \leq 2$. We bound the probability of the latter event crudely: for it to occur there must have been at least $t_0-1$ of the $2t_0$ increments that take the value $-1$, hence
	$$
	\prob(N_{2t} = 2 \, \mid \, X_0 = 0) \leq {2t_0 \choose t_0-1} (2t_0/\m)^{t_0-1}
	\leq C \m^{-t_0+1} \, ,
	$$
where $C=C(t_0) > 0$. This and the previous estimate concludes the proof of the bound on $\p^{2t}(0,e_{1,1})$.

The equality $\p^3(0,e_1) = [\m(\m-1)]^{-1}$ stems from the fact that there are precisely $(\m-1)$ non-backtracking paths of length $3$ from $0$ to $e_1$, and the probability of taking each is $[\m(\m-1)^2]^{-1}$. The bound on $\p^{2t+1}(0,e_1)$ is performed almost identically to the bound on $\p^{2t}(0,e_{1,1})$, we omit the details.
\end{proof}

We prove recursive bounds on $\E|\partial B(k)|$
that form the key ingredient in the proof of Proposition \ref{prop-unlacing}.
Before doing so, we recall some notation.
For a subset of vertices $A$, we say that an event $\M$ occurs {\em off $A$},
intuitively, if it occurs in $G_p \setminus A$. Formally, for a percolation configuration
$\omega$ and a set of vertices $A$, we write $\omega_A$ for the configuration obtained from $\omega$ by
turning all the edges touching $A$ to closed. The event ``$\M$ occurs off $A$''
is defined to be $\{\omega \colon \omega_A \in \M\}$.
We often drop $p$ from the notation when it is clear what $p$ is. This framework also allows us to address the case when $A=A(\omega)$ is a \emph{random} set measurable
with respect to $G_p$, the most prominent example being $A= B_0(r)$ for some $r\geq 1$. In this case,
the event $\{\M \hbox{ {\rm occurs off} } A(\omega)\}$ is defined to be
   \be\label{offdef}
    \{\M \hbox{ {\rm occurs off} } A(\omega)\} = \{\omega \colon \omega_{A(\omega)} \in \M \} \, .
    \ee
For this example, we shall rely on the fact that, for an arbitrary event $\M$ and $A=B_x(s)$
(see \cite[(3.1)]{HN12}),
    \be
    \label{offdescription}
    \prob(\M \off B_x(s)) =
    \sum_{A} \prob(B_x(s)=A) \prob(\M \off A) \, ,
    \ee
For two events $E$, $F$, we let$E\circ F$ denote the event that there exists a set of bonds $B$ such that
$\omega_B\in E$ and $\omega_{B^c}\in F$ (here we abuse notation slightly and now use $\omega_B$
for the configuration obtained from $\omega$ by
turning all the edges in $B$ to closed).
Then, the BKR-inequality states that
	\eqn{
	\lbeq{BK-def}
	\prob_p(E\circ F)\leq \prob_p(E)\prob_p(F).
	}

\begin{lemma}[Recursive bounds on $\E|\partial B(k)|$]
\label{lem-IH} For any $c>0$ there exists $B>0$ such that if
	$$
	p = {1 + {5/(2 m^2)} + B/m^{3} \over m-1} \, ,
	$$
then, \ch{for $\m=\m(B)$ sufficiently large and} 
for any $k \geq 1$ satisfying $\E |B(k)|\leq 2^{\m/2}/\m^3$,
	\eqn{
	\label{aim-IH}
	\E|\partial B(k)| \geq [1 + c/m^{3}]\E|\partial B(k-1)|.
	}
\end{lemma}

\proof We prove the claim by induction on $k$. Given $c>0$ we will choose $B$ to be large at the end of the proof --- this choice will not depend on $k$ or $m$. \ch{Given $B$, we choose $\m$ so large that
	$$
	p=(1 + {5/(2 m^2)} + B/m^{3})/(m-1)\leq (1+10/m^2)/(\m-1),
	$$
so our upper bound on $p$ is independent of $B$.} 

We start by initializing the induction. We have that $\E|\partial B(1)|=\m p$, while $\E|\partial B(0)|=1$, so that indeed \eqref{aim-IH} holds for $k=1$ (for any $B>0$ and $c>0$). Let $k \geq 1$ such that $\E|B(k)|\leq 2^{\m/2}/\m^3$ and assume the induction hypothesis holds for any $\ell \leq k-1$.

We will now estimate the conditional expectation of $|\partial B(k)|$ given $B(k-1)$. To be precise, when we condition on $B(k-1)$ we condition on all the open and closed edges touching a vertex of $B(k-2)$ (observe that since the graph is bipartite there cannot be two vertices of $\partial B(k-1)$ that are connected by an edge). This allows us to calculate $B(k-1)$ and note that edges from $\partial B(k-1)$ to $\partial B(k)$ are not revealed. Given this information, for each vertex $x \in \partial B(k-1)$ the number of edges that we have not revealed any information on is precisely $m - \sum_{y : y \sim x} \indic{y \in B(k-2)}$. Hence,
\begin{eqnarray*}
	\E \big [|\partial B(k)| \big |  B(k-1) \big ]
	\geq m p |\partial B(k-1)|  &-& p\sum_{(x,y)\in \Edges} \indic{x\in \partial B(k-1), y\in B(k-2)} \\
	&-&\frac{p^2}{2} \sum_{x,y \colon d(x,y)=2} \indic{x,y\in \partial B(k-1)} \, ,
\end{eqnarray*}
where the last term comes from subtracting the vertices of $|\partial B(k)|$ we counted more than once, which happens if they have more than one ``ancestor'' in $\partial B(k-1)$.

We take expectations in both sides and bound the two subtracted sums. We split the first sum according to whether the edge $(x,y)$ is open or not. If $(x,y)$ is open, then we must have that $y \in \partial B(k-2)$ and that $(x,y)$ is open off $B(k-2)$ (in the sense of \eqref{offdef}). Otherwise, that is, if $x \in \partial B(k-1), y \in B(k-2)$ and the edge $(x,y)$ is closed, then there exists $\ell \leq k-3$ and a vertex $w$ such that $\{0 \stackrel{=\ell}{\lrfill} w\} \circ P_{k-l-1}(w,x)\circ \{w \stackrel{k-\ell-2}{\lrfill} y\}$ occurs, where $P_{n}(w,x)$ denotes the event that there exists a path of occupied bonds of length $n$ connecting $w$ and $x$ (this is a monotone event). Indeed, let $\gamma_x$ and $\gamma_y$ be two shortest paths connecting $0$ to $x$ and $y$, respectively and take $w$ to be the last intersection of these paths and $\ell$ to be index of $w$ in the paths (this has to be the same number for both paths since they are shortest paths). Note that since the edge $(x,y)$ is closed, $w$ has to be at distance at most $k-3$ from $0$. Then the witness for the first event is $B(\ell)$ (that is, all the open and closed edges touching $B(\ell-1)$ and the two other witnesses are the parts of $\gamma_x$ and $\gamma_y$ starting at $w$ and ending at $x$ and $y$, respectively.
Similarly for the second contribution, if $x,y \in \partial B(k-1)$, then there must exists $\ell \leq k-2$ and a vertex $w$ such that $\{0 \stackrel{=\ell}{\lrfill} w\} \circ P_{k-l-1}(w,x) \circ P_{k-l-1}(w,y)$ occurs.

We sum over $\ell$ and $w$ and use the BK-Reimer inequality to obtain that
	\eqan{
	\label{basic-split}
	\E|\partial B(k)| &\geq \m p \E|\partial B(k-1)| - \text{(*)} - \text{(I)} -\text{(II)} \, , }
where
	$$
	\text{(*)} = p \sum_{w,x: x \sim w} \prob(0 \stackrel{=k-2}{\lrfill} w, (x,w) \text{ is open off } B(k-2)) \, ,
	$$
and
	$$
	\text{(I)} = p\sum_{\ell=0}^{k-3}\sum_{w,x,y\colon x \sim y}
	\prob(0 \stackrel{=\ell}{\lrfill} w) \prob(P_{k-\ell-1}(w,x)\circ \{w \stackrel{k-\ell-2}{\lrfill} y\}) \, ,
	$$
and
	$$
	\text{(II)} = \frac{p^2}{2}\sum_{\ell=0}^{k-2} \sum_{w,x,y\colon d(x,y)=2}
	\prob(0 \stackrel{=\ell}{\lrfill} w) \prob(P_{k-\ell-1}(w,x) \circ P_{k-\ell-1}(w,y)) \, .
	$$
We start by bounding from above the sum (*). By conditioning on $B(k-2)$ we may rewrite (*) as
	$$
	\text{(*)} = p \sum_{w} \sum_{A: 0 \stackrel{=k-2}{\lrfill} w} \prob(B(k-2)=A)
	\sum_{x\colon x \sim w}\prob((x,w) \text{ is open off } A \mid B(k-2)=A) \, .
	$$
Note that the probability that $(x,w)$ is open off $A$ equals $p$ for any $x$ such that $x \sim w$ and the edge $(x,w)$ is not in $A$. Since $A$ is such that $0 \stackrel{=\ell}{\lrfill} w$ and $\prob(B(k-2)=A)>0$, we learn that there are at most $m-1$ such possible $x$'s (instead of $m$, the total number of neighbors of $w$). Hence
	\be
	\label{boundstar}
	\text{(*)} \leq (m-1)p^2 \E | \partial B(k-2)| \leq (m-1)p^2 \E | \partial B(k-1)| \, ,
	\ee
where in the last line we used the induction hypothesis. We proceed by bounding from above the two sums (I) and (II). We handle the sums separately according to whether $k-\ell-1 \leq \mnot$ or not. To that aim, we define
	$$ \text{(I)}_1
	= p\sum_{\ell=0}^{k-1-\mnot}\sum_{w,x,y\colon x \sim y}
	\prob(0 \stackrel{=\ell}{\lrfill} w) \prob(P_{k-\ell-1}(w,x)\circ \{w \stackrel{k-\ell-2}{\lrfill} y\}) \, ,
	$$
and
	$$
	\text{(I)}_2 = p\sum_{\ell=k-1-\mnot}^{k-2}\sum_{w,x,y\colon x \sim y}
	\prob(0 \stackrel{=\ell}{\lrfill} w) \prob(P_{k-\ell-1}(w,x)\circ \{w \stackrel{k-\ell-2}{\lrfill} y\}) \, ,
	$$
and similarly we define (II)$_1$ and (II)$_2$. Our convention is that if $k-1 \leq \mnot$, then (I)$_2=$(II)$_2=0$.

It turns out that (I)$_1$ and (II)$_1$ contribute a negligible amount to (\ref{basic-split}). Indeed, when $k-\ell-1 \geq \mnot$ we use Lemma \ref{moreunifconnbd} to bound
	$$
	\prob(P_{k-\ell-1}(w,x)) \leq {C\E|B(k-\ell-1)|\over 2^m}\, .
	$$
We use this estimate and the BK inequality to bound
	$$
	\text{(I)}_1 \leq C mp 2^{-m} \sum_{\ell\leq k-1 -\mnot} \E|\partial B(\ell)|
	\big ( \E|B(k-\ell-1)| \big )^2 \, .
	$$
We bound $\E|B(k-\ell-1)| \leq \E|B(k)| \leq m^{-3}2^{m/2}$ by our assumption on $k$ to get that
	$$
	\text{(I)}_1 \leq C m^{-5} p  \sum_{\ell\leq k-1 -\mnot} \E|\partial B(\ell)|
	\leq C m^{-5} p \sum_{\ell\leq k-1 -\mnot} {\E|\partial B(k-1)| \over [1+cm^{-3}]^{k-\ell-1}} \, ,
	$$
where the last inequality is due to our induction hypothesis. This yields the bound
	\be
	\label{boundI1}
	\text{(I)}_1 \leq Cm^{-2} p \E|\partial B(k-1)| \leq Cm^{-3} \E|\partial B(k-1)| \, ,
	\ee
since $p = O(m^{-1})$ and $C>0$ may depend on $c$. An almost identical calculation gives that
	\be
	\label{boundII1}
	\text{(II)}_1 \leq Cm^{-3} \E|\partial B(k-1)| \, .
	\ee

Bounding (I)$_2$ and (II)$_2$ is more delicate and the local structure of the hypercube comes into play. Let us start with bounding (II)$_2$ since it is slightly simpler. We start by bounding
	$$
	\prob(P_{k-\ell-1}(w,x)\circ P_{k-\ell-1}(w,y))
	\leq m (m-1)^{2k-2\ell-2} p^{2k-2\ell-2} \p^{k-\ell-1,k-\ell-1}(x,w,y) \, ,
	$$
where $\p^{t_1,t_2}(x,w,y)$ is the probability that a non-backtracking random walk starting from $x$ visits $w$ at time $t_1$ and visits $y$ at time $t_1+t_2$. The reason for this bound is that if the event on the left hand side occurs, then there exists a simple open path of length precisely $2k-2\ell-2$ from $x$ to $y$ going through $w$ at time $k-\ell-1$. The number of such paths is bounded above by $m (m-1)^{2k-2\ell-2} \p^{k-\ell-1,k-\ell-1}(x,w,y)$ and the estimate follows by the union bound.
By transitivity we get that
	$$
	\text{(II)}_2 \leq {(1+O(m^{-1}))p^2 \over 2} \sum_{\ell=k-1-\mnot}^{k-2}
	\E |\partial B(\ell)| \sum_{x, y \colon d(x,y)=2} [(m-1)p]^{2k-2\ell-2} \p^{k-\ell-1,k-\ell-1}(x,0,y) \, .
	$$
Note that the sum over $x,y$ in the right hand side does not depend on the $0$, so by we may rewrite this sum as
	\begin{eqnarray*}
	&&\hskip-3cm\sum_{x, y \colon d(x,y)=2} [(m-1)p]^{2k-2\ell-2} \p^{k-\ell-1,k-\ell-1}(x,0,y)\\
	&=& 2^{-m} \sum_{x, y\colon d(x,y)=2} \sum_v [(m-1)p]^{2k-2\ell-2} \p^{k-\ell-1,k-\ell-1}(x,v,y) \\
	&=& 2^{-m} [(m-1)p]^{2k-2\ell-2} \sum_{x, y\colon d(x,y)=2} \p^{2k-2\ell-2}(x,y) \\
	&=& {m(m-1) \over 2} [(m-1)p]^{2k-2\ell-2} \p^{2k-2\ell-2}(0,e_{1,1}) \, ,
	\end{eqnarray*}
where in the last inequality we used the fact that on the hypercube, $\p^t(x,y)$ is the same for any pair $x,y$ such that $d(x,y)=2$, and $e_{1,1}$ is the hypercube vector $(1,1,0,\ldots,0)$. We now use the induction hypothesis which implies that $\E |\partial B(\ell)| \leq \E|\partial B(k-1)|$  to get the bound of
	$$
	\text{(II)}_2 \leq {(1+O(m^{-1}))p^2 m(m-1) \E|\partial B(k-1)| \over 4}
	\sum_{t=1}^{\mnot} [(m-1)p]^{2t} \p^{2t}(0,e_{1,1}) \, .
	$$
We now appeal to Lemma \ref{nbw1} and use the fact that $\mnot = O(\m \log \m)$ and that $\ch{(m-1)p \leq 1+10/m^2}$. A straightforward calculation with these gives that
	\be
	\label{O1}
	\sum_{t=1}^{\mnot} [(m-1)p]^{2t} \p^{2t}(0,e_{1,1}) = {2 + O(m^{-1}) \over m(m-1)}  \, .
	\ee
Thus,
	\be
	\label{boundII2}
	\text{(II)}_2 \leq {p^2 \over 2}(1+O(m^{-1}))\E|\partial B(k-1)|  \, .
	\ee

We proceed with bounding (I)$_2$. We begin by estimating
	$$
	\prob(P_{k-\ell-1}(w,x)\circ \{w \stackrel{k-\ell-2}{\lrfill} y\})
	\leq \sum _{s=0}^{k-\ell-2} \prob(P_{k-\ell-1}(w,x)\circ P_s(w,y)) \, ,
	$$
and further bound, for each $s$,
	$$
	\prob(P_{k-\ell-1}(w,x)\circ P_s(w,y) \})
	\leq m(m-1)^{k-\ell-1+s-1} p^{k-\ell-1+s} \p^{k-\ell-1,s}(x,w,y) \, ,
	$$
where $\p^{t_1,t_2}(x,w,y)$ was defined earlier, and the reasoning for this bound is as before. We get
	\be
	\label{O2}
	\text{(I)}_2 \leq (1+O(m^{-1}))p\sum_{\ell=k-1-\mnot}^{k-3}
	\E|\partial B(l)| \sum_{s=0}^{k-\ell-2} [(m-1)p]^{k-\ell-1+s}\sum_{x,y\colon x \sim y}
	\p^{k-\ell-1,s}(x,0,y) \, .
	\ee
As before, the sum over $x,y$ does not depend on $0$, that is
	$$
	\sum_{x,y\colon x \sim y} \p^{k-\ell-1,s}(x,0,y)
	= 2^{-m} \sum_{v} \sum_{x,y\colon x \sim y} \p^{k-\ell-1,s}(x,v,y)
	= 2^{-m} \sum_{x,y\colon x \sim y} \p^{k-\ell-1+s}(x,y)
	= m \p^{k-\ell-1+s}(0,e_1) \, ,
	$$
where $e_1$ is just the vector $(1,0,\ldots,0)$. As before we use the induction hypothesis to derive that $\E|\partial B(\ell)| \leq \E |\partial B(k-1)|$ to get that
	\be
	\label{O3}
	\text{(I)}_2 \leq (1+O(m^{-1}))
	mp \E |\partial B(k-1)| \sum_{\ell=k-1-\mnot}^{k-3}
	\sum_{s=0}^{k-\ell-2} [(m-1)p]^{k-\ell-1+s} \p^{k-\ell-1+s}(0,e_1) \, .
	\ee
A straightforward manipulation with the double sum gives that
	$$
	\text{(I)}_2 \leq (1+O(m^{-1})) mp \E |\partial B(k-1)|
	\sum_{t=3}^{2\mnot} 2 \lfloor t/2 \rfloor [(m-1)p]^{t}  \p^{t}(0,e_1) \, .
	$$
The dominant term here is $t=3$. We appeal to Lemma \ref{nbw1}, the fact that $\mnot = O(\m \log \m)$ and that $(m-1)p=1+O(m^{-2})$ to obtain that
	\be
	\label{O4}
	\sum_{t=3}^{2\mnot} 2 \lfloor t/2 \rfloor [(m-1)p]^{t}  \p^{t}(0,e_1)
	= {2 + O(m^{-1}) \over m(m-1) } \, .
	\ee
We get the bound
	\be
	\label{O5}
	\text{(I)}_2 \leq (1+O(m^{-1})) mp \E |\partial B(k-1)| [2m^{-2} + O(m^{-3})] \, .
	\ee
Finally, we put this together with (\ref{boundI1}), (\ref{boundII1}) and (\ref{boundII2}) into (\ref{basic-split}) to obtain
	\be
	\label{O6}
	\E|\partial B(k)| \geq \big [ mp - (m-1)p^2 - 2m^{-1}p - p^2/2 - Cm^{-3} \big ]\E|\partial B(k-1)| \, .
	\ee
\ch{Here we stress that since $p\leq (1+10/\m^2)/(m-1)$, 
the constants hidden in the $O(\cdot)$ in \eqref{O1}-\eqref{O6} are \emph{independent} of the constant
$B$ from the definition of $p$, which implies that also 
$C>0$ may depend on $c>0$ but not on $B$.}
We plug in the value of $p$ and a straightforward calculation shows that we can choose $B>0$ large enough such that
	$$
	mp - (m-1)p^2 - 2m^{-1}p - p^2/2 - Cm^{-3} \geq 1 + cm^{-3} \, ,
	$$
concluding our proof.

\qed

\noindent
{\it Proof of Proposition \ref{prop-unlacing}.}
By Lemma \ref{lem-IH}, $\E|\partial B(k)|\geq (1+c/\m^3)^k$ as long as
$\expec|B(k)|\leq 2^{\m/2}/\m^3$. Hence, eventually $\expec|B(k)|\geq 2^{\m/2}/\m^3$, proving the claim.
\qed




\begin{thebibliography}{99}


\bibitem{AN} Aizenman M.\ and Newman C.M.\ (1984), Tree graph
inequalities and critical behavior in percolation models, J.
Stat. Phys., 44: 393-454.

\bibitem{AKS82} Ajtai M., Koml\' os J.\ and Szemer\'edi E. (1982), Largest random component of a $k$-cube. Combinatorica {2}(1): 1--7.

\bibitem{Aldo97} Aldous, D.\ (1997).
Brownian excursions, critical random graphs and the multiplicative coalescent.
Ann.\ Probab. {25}: 812--854.








\bibitem{BA91} Barsky D.\ and  Aizenman M.\ (1991), Percolation critical exponents under the triangle condition. Ann. Probab.  19(4): 1520--1536.



\bibitem{BerKes85}
van~den Berg, J.\ and Kesten, H.\ (1985),
Inequalities with applications to percolation and reliability.
J.\ Appl.\ Prob.\ 22: 556--569.

\bibitem{B1} Bollob\' as B. (1984), The evolution of random graphs.
Trans.\ Amer.\ Math.\ Soc.\ 286: 257--274.

\bibitem{Bol85} Bollob\'as, B. (1985), {\em Random Graphs.}
Academic Press, London.

\bibitem{BolKohLuc92}
Bollob\' as B., Kohayakawa Y. and \L uczak T. (1992),
The evolution of random subgraphs of the cube.
Random Structures Algorithms 3(1): 55--90.

\bibitem{BCHSS1} Borgs C., Chayes J.T., van der Hofstad R., Slade G. and Spencer J. (2005),
Random subgraphs of finite graphs: I. The scaling window under the triangle condition.
Random Structures Algorithms 27: 137--184.

\bibitem{BCHSS2} Borgs C., Chayes J.T., van der Hofstad R., Slade G. and Spencer J. (2005),
Random subgraphs of finite graphs: II. The lace expansion and the
triangle condition. Ann.\ Probab.\ 33: 1886-1944.

\bibitem{BCHSS3} Borgs C., Chayes J.T., van der Hofstad R., Slade G. and Spencer J. (2006),
Random subgraphs of finite graphs: III.  The phase transition for
the $n$-cube. Combinatorica 26: 395-410.


\bibitem{ReimerRef} Borgs C.,  Chayes, J. T. and  Randall D. (1999),
The van den Berg-Kesten-Reimer inequality: a review,
{\em Perplexing problems in probability}, 159--173,
Progr. Probab., 44, Birkhh\" auser Boston, Boston, MA.



\bibitem{CliLiaSla07}
Clisby N., Liang R., and Slade G.\ (2007),
\newblock Self-avoiding walk enumeration via the lace expansion.
\newblock {\em J. Phys. A}, {40}(36):10973--11017.

\bibitem{CliSla09}
Clisby N.\ and Slade G.\ (2009),
\newblock Polygons and the lace expansion.
\newblock In {\em Polygons, polyominoes and polycubes}, volume~{775} of
  {\em Lecture Notes in Phys.}, pages 117--142. Springer, Dordrecht.





\bibitem{ER} Erd\H{o}s P. and R\'enyi A. (1960),
On the evolution of random graphs,
Magyar Tud.\ Akad.\ Mat.\
Kutat\'o Int. K\H{o}zl. 5: 17--61.

\bibitem{ES} Erd\H{o}s P. and Spencer J. (1979), Evolution of the $n$-cube, Comput. Math. Appl. {5}, 33-39.



\bibitem{FilPem93}
Fill J.A.\ and Pemantle, R.\ (1993)
\newblock Percolation, first-passage percolation and covering times for
  {R}ichardson's model on the {$n$}-cube.
\newblock {\em Ann. Appl. Probab.}, {3}(2):593--629.

\bibitem{FitHof11a} Fitzner R. and van der Hofstad R. (2011),
Non-backtracking random walk. Preprint. Available at
\url{http://www.win.tue.nl/~rhofstad/NBWfin.pdf}.


\bibitem{GauRus78}
Gaunt D.S.\ and Ruskin H.\ (1978), Bond percolation processes in $d$ dimensions. J.\ Phys.\ A: Math.\ Gen., {11}:1369–-1380.

\bibitem{Grim99} Grimmett G. (1999), {\em Percolation}.
Second edition. Grundlehren der Mathematischen Wissenschaften
[Fundamental Principles of Mathematical Sciences], 321. Springer-Verlag, Berlin.



\bibitem{HS90} Hara T. and Slade G. (1990), Mean-field critical
behaviour for percolation in high dimensions. Comm. Math.
Phys. 128: 333-391.


\bibitem{HH} Heydenreich M. and van der Hofstad R. (2007), Random graph
asymptotics on high-dimensional tori. Comm. Math. Phys. 270(2): 335-358.

\bibitem{HH2} Heydenreich M. and van der Hofstad R. (2011),
Random graph asymptotics on high-dimensional tori II.
Probability Theory and Related Fields, 149(3-4): 397-415.

%

\bibitem{HN12}
van der Hofstad R.\ and Nachmias A.\ (2012). Hypercube percolation. Preprint.

\bibitem{HS05} van der Hofstad, R. and  Slade G. (2005),
Asymptotic expansions in $n\sp {-1}$ for percolation critical values on the $n$-cube
and $\Bbb Z\sp n$. Random Structures Algorithms 27(3): 331--357.

\bibitem{HS06} van der Hofstad R. and  Slade G. (2006),
Expansion in $n\sp {-1}$ for percolation critical values on the $n$-cube and
$\Bbb Z\sp n$: the first three terms. Combin. Probab. Comput. 15(5): 695--713.


\bibitem{JKL93}
Janson S., Knuth D.E.,{\L}uczak T. and Pittel, B. (1993),
The birth of the giant component.
Random Structures and Algorithms 4: 71--84.

\bibitem{JLR00}
Janson S., {\L}uczak T. and Rucinski, A. (2000),
\newblock {\em Random Graphs}.
\newblock Wiley, New York.




\bibitem{KN08} Kozma G. and Nachmias A. (2009),
The Alexander-Orbach conjecture holds in high dimensions,
  Invent. Math. 178(3): 635--654.


\bibitem{KN11} Kozma G. and Nachmias A. (2011),
Arm exponents in high-dimensional percolation,
Journal of the American Mathematical Society 24: 375-409.


\bibitem{LPW09} Levin D.,  Peres, Y. and  Wilmer, E. (2009),
{\it Markov Chains and Mixing Times.}
With a chapter by James G. Propp and David B. Wilson.
American Mathematical Society, Providence, RI.

\bibitem{Lu} \L uczak T. (1990), Component behavior near the critical
point of the random graph process. Random Structures Algorithms 1:  287--310.

\bibitem{LPW94}
{\L}uczak T., Pittel B. and Wierman, J.C.  (1994),
The structure of a random graph at the point of the phase transition.
Trans.\ Amer.\ Math.\ Soc. 341:721--748.










\bibitem{NP3} Nachmias A. and Peres Y. (2008),
Critical random graphs: diameter and mixing time,
Ann.\ Probab.\ 36(4): 1267-1286.

\bibitem{Penrose98} Penrose M. D. (1998), Random minimal spanning tree and percolation on the $N$-cube,
{\em Random Structures Algorithms} {\bf 12}, no. 1, 63--82.



\bibitem{PitWor05}
Pittel, B. and Wormald, N. (2005),
Counting connected graphs inside-out.
J.\ Combin.\ Theory Ser.\ B, 93(2):127--172.

\bibitem{Reim00}
Reimer, D. (2000),
\newblock Proof of the van den {B}erg-{K}esten conjecture.
\newblock Combin. Probab. Comput. 9(1):27--32.

\end{thebibliography}
\end{document}